\documentclass[10pt,reqno]{amsart}

\usepackage{amssymb}
\usepackage{amsthm}
\usepackage{amsmath}
\usepackage{enumerate}

\title{The tilted Carath{\'e}odory class and its applications }
\author[L.-M. Wang]{Li-Mei Wang}
\address{Division of Mathematics, Graduate School of Information Sciences, Tohoku University, 6-3-09 Aramaki-Aza-Aoba, Aoba-ku, Sendai, Miyagi 980-8579, Japan, phone : +81-22-795-4636}
\email{rime@ims.is.tohoku.ac.jp}
\date{}


\newtheorem{thm}{Theorem}
\newtheorem{cor}{Corollary}
\newtheorem{lem}{Lemma}
\newtheorem{pro}{Proposition}
\theoremstyle{definition}
\newtheorem{definition}{Definition}
\newtheorem{remark}{Remark}

\renewcommand{\Re}{{\operatorname{Re}\,}}
\renewcommand{\Im}{{\operatorname{Im}\,}}
\renewcommand{\arccos}{{\operatorname{arccos}\,}}
\newcommand{\Ext}{{\operatorname{Ext}\,}}

\begin{document}

\maketitle

\begin{abstract}
This survey mainly deals with the tilted Carath{\'e}odory class by angle $\lambda$ (denoted by
$\mathcal{P}_{\lambda}$) an element of which maps the unit disc into the tilted right half-plane $\{w: \Re e^{i\lambda} w>0\}$. Firstly we will characterize $\mathcal{P}_{\lambda}$ from different aspects. In section 3, various estimates of functionals over $\mathcal{P}_{\lambda}$ are deduced from the known corresponding estimates of $\mathcal{P}_{0}$ or elementary functional analysis. Finally some subsets of analytic functions related to $\mathcal{P}_{\lambda}$ including close-to-convex functions with argument $\lambda$, $\lambda$-spirallike functions, $\lambda$-Robertson functions and analytic functions whose derivative is in $\mathcal{P}_{\lambda}$ are also considered as applications.  
\end{abstract}

\section{Introduction}
\par Let $\mathcal{A}$ be the family of functions $f$ analytic in the unit disc $\mathbb{D}=\{z\in \mathbb{C}:\, |z|<1\}$, and $\mathcal{A}_{1}$ be the subset of $\mathcal{A}$ consisting of functions $f$ which are normalized by $f(0)=f'(0)-1=0$ while $\mathcal{A}_{0}$ with normalization $f(0)=1$. A function $f\in\mathcal{A}$ is said to be subordinate to a function $F\in\mathcal{A}$ (in symbols $f\prec F$ or $f(z)\prec F(z)$) in $\mathbb{D}$ if there exists an analytic function $\omega(z)$ with $|\omega(z)|<1$ and $\omega(0)=0$, such that 
\[
f(z)=F(\omega(z))
\]
in $\mathbb{D}$. When $F$ is a univalent function, the condition $f\prec F$ is equivalent to $f(\mathbb{D})\subseteq F(\mathbb{D})$ and $f(0)=F(0)$.
\par \par Let 
\[
\mathcal{P}_{\lambda}=\left\{p\in\mathcal{A}_{0} :\,  \Re e^{i \lambda}p(z)>0 \right\}.
\]
Here and hereafter we always suppose $-\pi/2<\lambda<\pi/2$. Note that $\mathcal{P}_{\lambda}$ is a convex and compact subset of $\mathcal{A}$ which is equipped with the topology of uniform convergence on compact subsets of $\mathbb{D}$. 
Since $\mathcal{P}_{0}$ is the well-known Carath{\'e}odory class, we call $\mathcal{P}_{\lambda}$ \textit{the tilted Carath{\'e}odory class by angle $\lambda$}. Also we let
\[
\mathcal{P}=\bigcup_{-\pi/2<\lambda<\pi/2} \mathcal{P}_{\lambda}.
\]
\par In this note, we always let
\[
p_{\lambda}(z)=\frac{1+e^{-2i\lambda}z}{1-z}.
\]
It is easy to see that $p_{\lambda}$ univalently maps the unit disk $\mathbb{D}$ onto $\mathbb{H}_{\lambda}=\{w\in\mathbb{C}:\, \Re e^{i\lambda}w>0\}$ which is called \textit{the tilted right half-plane by angle $\lambda$}. Later we will see that this function plays an important role while investigating the properties of $\mathcal{P}_{\lambda}$.
\par The Carath{\'e}odory class $\mathcal{P}_{0}$ occupies an extremely important place in the theory of functions and has been studied by many authors ([5, Chapter 7, p. 77], \cite{GrahamKohr}, \cite{Kirwan}, \cite{Libera}, \cite{Robertson1}, \cite{Robertson2}, \cite{Ruscheweyh}, \cite{Sakaguchi}, \cite{Zmorovic1}, \cite{Zmorovic2}). The tilted Carath{\'e}odory class $\mathcal{P}_{\lambda}$ scatters in some papers (see \cite{KimSugawa2006}, \cite{KimSugawa}, \cite{Libera}, \cite{Robertson3}), although the name was not given in the literature.
\par Section 2 is devoted to characterizations of the functions belonging to the class $\mathcal{P}_{\lambda}$ from different aspects. A linear relation between the elements of $\mathcal{P}_{\lambda}$ and $\mathcal{P}_{0}$ implies that the functions in $\mathcal{P}_{\lambda}$ can be described in terms of integral and  subordination. We show also that $\mathcal{P}_{\lambda}$ can be regarded as the dual and the second dual sets of some analytic function famlilies. 
\par In section 3, the extremal functions of $\mathcal{P}_{\lambda}$ are deduced directly from those of the Carath{\'e}odory class $\mathcal{P}_{0}$. With the aid of these extremal functions, the sharp estimates of some functionals over $\mathcal{P}_{\lambda}$, for instance, the $n$-th coefficient functional, the distortion and growth functionals, have been obtained. For the cases of other functionals, we summarize some other methods to deal with the extremal problems. By using fundamental functional analysis, we obtain the sharp estimate of $\Re zp'(z)/p(z)$ with $p(z)\in\mathcal{P}_{\lambda}$. The estimate of this entity was considered by some authors, see Bernardi \cite{Bernardi} and Robertson \cite{Robertson2} , and the sharp one appeared in a paper of Ruscheweyh and Singh \cite{Ruscheweyh}. Although our estimate is equivalent to that in \cite{Ruscheweyh}, the extremal functions which make the estimate sharp are also given explicitly. 
\par The last section is concerned with some applications of our results to Geometric Function Theory. We will consider $\lambda$-spirallike functions, close-to-convex functions with argument $\lambda$, $\lambda$-Robertson functions and analytic functions whose derivative is in $\mathcal{P}_{\lambda}$.

\section{Characterisations of $\mathcal{P}_{\lambda}$}
In this section, we list some characterizations of $\mathcal{P}_{\lambda}$ for later use. Before proceeding to it, we shall first introduce some notations. The convolution (or Hadamard product) $f*g$ of two functions $f,\, g\in \mathcal{A}$ with series expansions $f(z)=\sum_{n=0}^{\infty}a_{n}z^n$ and $g(z)=\sum_{n=0}^{\infty}b_{n}z^n$ is defined by
\[
f*g(z)=\sum_{n=0}^{\infty}a_{n}b_{n}z^n.
\]
Obviously, we have $f*g\in \mathcal{A}$. For an introduction to the theory of convolutions in the present context we refer to \cite{Ru}. For an analytic function $h\in\mathcal{A}_{0}$, we note that
\begin{equation}\label{h*p}
\begin{split}
h(z)*\frac{1+Az}{1-z}
&=h(z)*\left(\frac{1+A}{1-z}-A \right)\\
&=(1+A)h(z)-A
\end{split}
\end{equation}
for a complex number $A$ which will be used several times. For a set $V\subset \mathcal{A}_{0}$, define the dual set
\[
V^{*}=\{g\in\mathcal{A}_{0}:\, (f*g)(z)\not=0\, \text{in}\, \mathbb{D}\, \text{for any}\, f\in V\},
\]
and $V^{**}=(V^{*})^{*}$, the second dual.
\par Here we recall that from the proof of Theorem 1.3 in \cite{Ru}, a function $h\in\mathcal{A}_{0}$ belongs to $\mathcal{P}$ if and only if 
\[
th(xz)+(1-t)h(yz)\not=0,
\]
for any $|x|=|y|=1$, $0\leq t\leq 1$ and $z\in\mathbb{D}$.
 
\begin{thm}\label{theorem 1}
Let $\lambda\in(-\pi/2,\pi/2)$ be a real constant, the following conditions are equivalent for a function $p\in\mathcal{A}_{0}$ 

\begin{enumerate}[(\@i)]
\makeatletter
\item $p\in \mathcal{P}_{\lambda}$; \\
\item  $\frac{e^{i\lambda}p-i\sin \lambda}{\cos \lambda }\in \mathcal{P}_{0}$;\\
\item  There exists a Borel probability measure $\mu$ on $\partial\mathbb{D}$ such that $p(z)$ can be represented by
\[
p(z)=\int_{\partial\mathbb{D}} \frac{1+e^{-2i\lambda}xz}{1-xz}d\mu(x);
\]
\item  $p\prec p_{\lambda}$ in $\mathbb{D}$;
\item  $p\in V_{\lambda}^{*}$, where
\[
V_{\lambda}=\left\{ \frac{1+\frac{1-e^{-2i\lambda}x}{(1+e^{-2i\lambda})x}z}{1-z}:\, |x|=1 \right\};
\]
\item $p\in W_{\lambda}^{**}$, where
\[
W_{\lambda}=\left\{t \frac{1+x e^{-2i\lambda}z}{1-xz}+(1-t) \frac{1+y e^{-2i\lambda}z}{1-yz}:\, |x|=|y|=1, \,0\leq t\leq 1\right\}.
\]
\end{enumerate}
\end{thm}

\begin{remark}
Theorem $\ref{theorem 1}$ implies that 
\[
\mathcal{P}_{\lambda}=V_{\lambda}^{*}
\]
and 
\[
\mathcal{P}_{\lambda}=W_{\lambda}^{**}.
\]
For $\mathcal{P}$, we have (see [28, Theroem 1.6])
\[
\mathcal{P}=\left\{\frac{1+xz}{1+yz}:\, |x|=|y|=1\right\}^{*}
\]
and
\[
\mathcal{P}=\left\{f\in\mathcal{A}_{0}:\, \Re f(z)>1/2 \right\}^{**}.
\]
\end{remark}

\noindent \textit{Proof of Theorem $\ref{theorem 1}$.} 
The equivalence of (i), (ii) and (iv) can be obtained immediately from the definition of $\mathcal{P}_{\lambda}$. The condition (iii) is reduced to the Herglotz integral representation of $\mathcal{P}_{0}$ (\cite{He}, see also \cite{DH}) when $\lambda=0$. Thus by the equivalence of (i) and (ii), we can easily get (i) $\Leftrightarrow$ (iii). Hence we only need to prove the equivalence of (i), (v) and (vi).
\par Firstly we will show (i) $\Leftrightarrow$ (v). It is sufficient to prove that
\begin{equation}\label{*}
p(z)* \frac{1+\frac{1-e^{-2i\lambda}x}{(1+e^{-2i\lambda})x}z}{1-z}\not=0
\end{equation}
for $|x|=1$ and $z\in \mathbb{D}$ if and only if $p\in\mathcal{P}_{\lambda}$.

By making use of equation $(\ref{h*p})$, we have ($\ref{*}$) is equivalent to 
\[
p(z)*\left(\frac{1}{1-z}-\frac{1-e^{-2i\lambda}x}{1+x}\right)\not=0, 
\]
namely,
\[
p(z)\not=\frac{1-e^{-2i\lambda}x}{1-x}
\]
for $|x|=1$ and $z\in \mathbb{D}$. Since the set $\left\{(1-e^{-2i\lambda}x)/(1-x):\, |x|=1 \right\}$ is equal to the line $\{w:\, \Re e^{i\lambda}w =0\}$, the above inequality implies that the image $p(\mathbb{D})$ lies in the tilted right half-plane by angle $\lambda$ by observing that $p(0)=1$ which is our assertion. 

\par Next we will show (i)$\Leftrightarrow$(vi). To this end, we need to prove that an analytic function $h\in\mathcal{A}_{0}$ satisfies 
\begin{equation}\label{h(z)}
h(z)*\left(t \frac{1+x e^{-2i\lambda}z}{1-xz}+(1-t) \frac{1+y e^{-2i\lambda}z}{1-yz}\right)\not=0
\end{equation}
for any $|x|=|y|=1$, $0\leq t\leq 1 $ and $z\in\mathbb{D}$, if and  only if 
\[
h*p(z)\not=0
\]
for $z\in\mathbb{D}$ and $p\in\mathcal{P}_{\lambda}$.
\par If an analytic function $h$ with $h(0)=1$ satisfies $(\ref{h(z)})$ for any $|x|=|y|=1$, $0\leq t\leq 1 $ and $z\in\mathbb{D}$, then by equation $(\ref{h*p})$, we have 
\[
th(xz)+(1-t)h(yz)\not=\frac{e^{-2i\lambda}}{1+e^{-2i\lambda}}
\]
for any $|x|=|y|=1$, $0\leq t\leq 1 $ and $z\in\mathbb{D}$. Thus by making use of the statements just before Theorem $\ref{theorem 1}$, we have there exists a real constant $\alpha\in[0,2\pi)$ such that
\[
\Re e^{i\alpha}\left(h(z)-\frac{e^{-2i\lambda}}{1+e^{-2i\lambda}}\right)>0,
\]
which implies that for any Borel probability measure $\mu$ on $ \partial\mathbb{D}$ we have
\[
\int_{\partial\mathbb{D}}\left(h(xz)-\frac{e^{-2i\lambda}}{1+e^{-2i\lambda}}\right)d\mu(x)\not=0.
\]
Hence for any Borel probability measure $\mu$ on $\partial\mathbb{D}$ we have 
\[
\int_{\partial\mathbb{D}}h(xz)d\mu(x)\not=\frac{e^{-2i\lambda}}{1+e^{-2i\lambda}},
\]
which is equivalent to
\[
h(z)*\int_{\partial\mathbb{D}}\frac{1+e^{-2i\lambda}xz}{1-xz}d\mu(x)\not=0.
\]
Then for any function $p\in \mathcal{P}_{\lambda}$ we have 
\[
p*h(z)\not=0
\]
for $z\in\mathbb{D}$ by the equivalence of (i) and (iii). Since the above procedure is invertible, we thus arrive at our conclusions.
\qed

\par By making use of the following Schur's lemma , we can get a convolution property of $\mathcal{P}_{\lambda}$.

\begin{lem}[see \cite{Schur} or \cite{Kortram}]\label{lemma 1}
Let $p_{1}(z)=1+\sum_{n=1}^{\infty}a_{n}z^n \in \mathcal{P}_{0}$ and $p_{2}(z)=1+\sum_{n=1}^{\infty}b_{n}z^n \in \mathcal{P}_{0}$, then 
\[
1+\sum_{n=1}^{\infty}\frac{a_{n}b_{n}}{2}z^n \in \mathcal{P}_{0}.
\]

\end{lem}

\begin{thm}\label{theorem 2}
 Let $p_{1}\in \mathcal{P}_{\lambda_{1}}$ and $p_{2}\in \mathcal{P}_{\lambda_{2}}$, then
\begin{equation}\label{*in2}
\Re e^{i(\lambda_{1}+\lambda_{2})}(p_{1}*p_{2})>-\cos(\lambda_{1}-\lambda_{2}).
\end{equation}
In particular, if $\cos(\lambda_{1}-\lambda_{2})<0$, then
\[
p_{1}*p_{2}\in \mathcal{P}_{\lambda_{1}+\lambda_{2}}.
\]
\end{thm}
\begin{proof}
Since $p_{1}(z)\in \mathcal{P}_{\lambda_{1}}$ and $p_{2}(z)\in \mathcal{P}_{\lambda_{2}}$, then by the equivalence (i) and (ii) in Theorem $\ref{theorem 1}$ we have
\[
\frac{e^{i\lambda_{1}}p_{1}(z)-i\sin\lambda_{1}}{\cos\lambda_{1}}\in \mathcal{P}_{0}
\]
and
\[
\frac{e^{i\lambda_{2}}p_{2}(z)-i\sin\lambda_{2}}{\cos\lambda_{2}}\in \mathcal{P}_{0}.
\]
Hence Schur's lemma implies that 
\[
\frac 1 2 \left(\frac{e^{i\lambda_{1}}p_{1}(z)-i\sin\lambda_{1}}{\cos\lambda_{1}}*\frac{e^{i\lambda_{2}}p_{2}(z)-i\sin\lambda_{2}}{\cos\lambda_{2}}\right)+\frac 1 2 \in \mathcal{P}_{0} 
\]
which is equivalent to (\ref{*in2}). 
\end{proof}


\section{Basic estimates of $\mathcal{P}_{\lambda}$}

\par The extremal points of $\mathcal{P}_{0}$ can be obtained by the Herglotz integral representation formula. A truly beautiful derivation of $\Ext \mathcal{P}_{0}$ was given by Holland \cite{Holland}, while Kortram \cite{Kortram} obtained it by elementary functional analysis. We state it as a lemma in order to get the corresponding result of $\mathcal{P}_{\lambda}$.

\begin{lem}\label{lemma 2}
\[
\Ext \mathcal{P}_{0}=\left\{\frac{1+xz}{1-xz},\, |x|=1\right\}.
\]
\end{lem}

\begin{thm}\label{theorem 3}
\[
\Ext \mathcal{P}_{\lambda}=\{p_{\lambda}(xz), \,|x|=1\}.
\]
\end{thm}
\begin{proof}
 A combination of Lemma $\ref{lemma 2}$ and the equivalence of (i) and (ii) in Theorem $\ref{theorem 1}$ implies our assertion. 
\end{proof} 

The next result which can be found in [7, p. 45] is a useful technique to solve extremal problems. If $\mathcal{F}$ is a convex subset of $\mathcal{A}$ and $J:\, \mathcal{A}\rightarrow \mathbb{R}$, then $J$ is called \textit{convex} on $\mathcal{F}$ provided that $J(tf+(1-t)g)\leq tJ(f)+(1-t)J(g)$ whenever $f,\, g \in\mathcal{F}$ and $0\leq t\leq 1$.

\begin{lem}\label{lemma 3}
Let $\mathcal{F}$ be a compact and convex subset of $\mathcal{A}$ and let $J$ be a real-valued, continuous, convex functional on $\mathcal{F}$. Then
\[
\max\{J(f):\, f\in\mathcal{F}\}=\max\{J(f):\, f\in\Ext\mathcal{F}\}.
\]
\end{lem}

\begin{thm}\label{theorem 4}
Let $p(z)=1+\sum_{n=1}^{\infty}p_{n}z^n \in \mathcal{P}_{\lambda}$ with $\lambda\in (-\pi/2,\pi/2)$, then
\[
|p_{n}|\leq 2\cos\lambda,
\]
and
\[
|p'(z)|\leq \frac{2\cos\lambda}{(1-r)^2},
\]
where $r=|z|<1$. The inequalities are sharp with extremal functions $p_{\lambda}(xz)$ where $|x|=1$.
\end{thm}
\begin{proof} 
It is easy to check that the above two functionals are real-valued, continuous and convex, thus by applying Lemma $\ref{lemma 3}$ the maximums of them are obtained over the reduced subset $\Ext \mathcal{P}_{\lambda}$, which is the set of $p_{\lambda}(xz)$ where $|x|=1$ in view of Lemma $\ref{lemma 2}$. Since
\[
p_{\lambda}(xz)=1+(1+e^{2i\lambda})\sum_{n=1}^{\infty}x^nz^n
\]
and
\[
p_{\lambda}'(xz)=\frac{1+e^{2i\lambda}}{(1-xz)^2},
\]
we complete the proof.
\end{proof}

\begin{thm}\label{theorem 5}
Let $p(z)\in \mathcal{P}_{\lambda}$ with $\lambda\in (-\pi/2,\pi/2)$, then
\[
\left|p(z)-\left(\frac{1}{1-r^2}+\frac{r^2}{1-r^2}e^{-2i\lambda}\right)\right|\leq\frac{2r\cos \lambda}{1-r^2}
\]
where $r=|z|<1$. In particular we have
\[
\frac{1+r^2\cos 2\lambda -2r\cos \lambda}{1-r^2}\leq \Re p(z)\leq \frac{1+r^2\cos 2\lambda +2r\cos \lambda}{1-r^2}
\]
and
\[
1/A(\lambda,r)\leq|p(z)|\leq A(\lambda,r)
\]
where $A(\lambda,r)$ is given by
\begin{equation}{\label{A(r)}}
A(\lambda,r)=\frac{\sqrt{({1-r^2})^2+4r^2\cos^2\lambda}+2r\cos\lambda}{1-r^2}.
\end{equation}
Those inequalities are sharp with extremal functions $p_{\lambda}(xz)$ where $|x|=1$.
\end{thm}
\begin{proof}
 By using the same arguments as in Theorem $\ref{theorem 4}$, the maximum of the first functional over $\mathcal{P}_{\lambda}$ is obtained over the set of $p_{\lambda}(xz)$ where $|x|=1$. Therefore
\begin{equation*}
\begin{split}
&\left|p_{\lambda}(xz)-\left(\frac{1}{1-r^2}+\frac{r^2}{1-r^2}e^{-2i\lambda}\right) \right|\\
&=\left|\frac{(1+e^{2i\lambda})(xz-r^2)}{(1-xz)(1-r^2)}\right|\\
&=\frac{2\cos \lambda}{1-r^2}\left| \frac{xz-r^2}{1-xz}\right|\\
&=\frac{2r\cos \lambda}{1-r^2}
\end{split}
\end{equation*}
where $r=|z|<1$ and $|x|=1$, which is what we want. The other estimates can be deduced directly from the first one.
\end{proof}

\begin{remark}
The first inequality in Theorem \ref{theorem 5} implies that the function $p\in\mathcal{P}_{\lambda}$ maps the disc $|z|<r<1$  into $U(\lambda, r)$ which is the hyperbolic disc in the tilted half plane $\mathbb{H}_{\lambda}$ centered at 1 with radius $ \arctan r$.
\end{remark}

\par Although the extremal problems of real-valued continuous linear functionals over $\mathcal{P}_{\lambda}$ can be solved within the set  $p_{\lambda}(xz)$, $|x|=1$, it is not applicable for general functionals. Robertson \cite{Robertson1} \cite{Robertson2} and Sakaguchi \cite{Sakaguchi} obtained variational formulae for $\mathcal{P}_{0}$ and showed that, for fixed $z\in\mathbb{D}$, the extremal values of 
\[
\Re F(p(z), zp'(z)), \,\,p\in\mathcal{P}_{0},
\]
where $F(u,v)$ is analytic in $(u,v)\in\mathbb{C}^2$, $\Re u>0$, are already attained by the functions
\[
t\left(\frac{1+xz}{1-xz}\right)+(1-t)\left(\frac{1+\bar{x}z}{1-\bar{x}z}\right), \,\, 0\leq t\leq 1,\,\, |x|=1.
\]
Since there exists a linear relation between $\mathcal{P}_{0}$ and $\mathcal{P}_{\lambda}$, we can see that for fixed $z\in\mathbb{D}$, the extremal values of 
\[
\Re F(p(z), zp'(z)), \,\,p\in\mathcal{P}_{\lambda},
\]
where $F(u,v)$ is analytic in $(u,v)\in\mathbb{C}^2$, $\Re e^{i\lambda}u>0$, are attained by the functions
\[
t\left(\frac{1+xe^{-2i\lambda}z}{1-xz}\right)+(1-t)\left(\frac{1+\bar{x}e^{-2i\lambda}z}{1-\bar{x}z}\right), \,\, 0\leq t\leq 1,\,\, |x|=1.
\]   
On the other hand, it follows from the Duality Principle [18, Corollary 1.1] and Remark 1 that the extremal values of 
\[ 
\frac{F_{1}(p)}{F_{2}(p)}
\]
where $F_{1}$ and $F_{2}$ are real-valued continuous linear functionals over $\mathcal{P}_{\lambda}$ with $F_{2}(P)\not =0$ for $p\in\mathcal{P}_{\lambda}$ are attained by a function in $W_{\lambda}$. It is clear that it is not easy to obtain extremal values even for those restricted classes of functions. Next we will give an estimate which can be solved by elementary functional analysis.

\begin{thm}\label{theorem 6}
\par Let $p\in \mathcal{P}_{\lambda}$ with $\lambda\in(-\pi/2,\pi/2)$, then
\[
\left|\frac{zp'(z)}{p(z)}\right|\leq M(\lambda,|z|),
\]
where 
\begin{equation}{\label{M2(r)}}
M(\lambda,r)=
\begin{cases}
\dfrac{2r\cos \lambda}{1+r^2-2r|\sin \lambda|}, \quad r<|\tan(\lambda/2)|,\\
\dfrac{2r}{1-r^2}, \quad r\geq |\tan(\lambda/2)|.\\
\end{cases}
\end{equation}
Equality holds for some point $z_{0}=re^{i\theta}$, $0<r<1$, if and only if $p(z)=p_{\lambda}(xz)$ where $x=e^{i(\alpha-\theta)}$ with $\alpha$ satisfying 
\begin{equation}\label{maximum point}
\begin{cases}
\alpha=\pi/2+\lambda, \quad r<-\tan(\lambda/2),\\
\alpha=-\pi/2+\lambda,\quad r<\tan(\lambda/2),\\
\sin(\alpha-\lambda)=-\frac{1+r^2}{1-r^2}\sin\lambda,\quad r\geq |\tan(\lambda/2)|.
\end{cases}
\end{equation}
\end{thm}

\begin{remark}
For fixed $0<r<1$, $M(\lambda,r)$ is a symmetric function in $\lambda$ with respect to the origin and it is also decreasing in $0\leq\lambda<\pi/2$. 
We thus have  $M(\lambda,r)\leq M(0,r)=2r/(1-r^2)$ for any $\lambda\in(-\pi/2,\pi/2)$ which is a known result for Gelfer functions ([5, p. 73], see also \cite{Yamashita}, \cite{Gelfer} or \cite{Sugawa}).
\end{remark}

\par In order to prove the above theorem, the following lemma is needed.

\begin{lem}\label{lemma 4}
\[
 N(\lambda,r)\leq \left|\frac{zp'_{\lambda}(z)}{p_{\lambda}(z)}\right|\leq M(\lambda,r)
\]
for $r=|z|<1$ where $M(\lambda,r)$, is defined by $(\ref{M2(r)})$ and 
\begin{equation}{\label{M1(r)}}
N(\lambda,r)=\frac{2r\cos \lambda}{1+r^2+2r|\sin \lambda|}.
\end{equation}
The upper equality holds at $z_{0}=re^{i\theta}$ with $\theta$ in place of $\alpha$ satisfying $\eqref{maximum point}$ while the lower one holds at $z_{0}=re^{i\theta}$ with $\theta$ satisfying
\begin{equation*}
\begin{cases}
\theta-\lambda=\pi/2, \quad \lambda>0,\\
\theta-\lambda=-\pi/2, \quad \lambda\leq 0.
\end{cases}
\end{equation*} 

\end{lem}
\begin{proof} 
By observing that
\[
\left|\frac{\overline{z}p'_{\lambda}(\overline{z})}{p_{\lambda}(\overline{z})} \right|=\left| \frac{zp'_{-\lambda}(z)}{p_{-\lambda}(z)}\right|,
\]
we can restrict to the case $\lambda\geq 0$.
\par Since $p_{\lambda}'(z)/p_{\lambda}(z)=(1+e^{2i\lambda})/(1-z)(e^{2i\lambda}+z)$, after letting $z=re^{i(\alpha+\lambda+\pi/2)}$ and $h(\alpha)=|(1-z)(e^{2i\lambda}+z)|^2=|(1-re^{i(\alpha+\lambda+\pi/2)})(e^{2i\lambda}+re^{i(\alpha+\lambda+\pi/2)})|^2$, we obtain
\[
h(\alpha)=(1+r^2+2r\sin(\alpha+\lambda))(1+r^2-2r\sin(\lambda-\alpha))
\]
and
\begin{equation}\label{h(a)}
\frac{2r\cos\lambda}{\max_{-\pi<\alpha\leq \pi}\sqrt{h(\alpha)}}\leq \left|\frac{zp_{\lambda}'(z)}{p_{\lambda}(z)}\right|\leq \frac{2r\cos\lambda}{\min_{-\pi<\alpha\leq \pi}\sqrt{h(\alpha)}}.
\end{equation}
It is thus sufficient to search for the maximum and minimum of $h(\alpha)$ over $-\pi<\alpha\leq \pi$. A simple calculation yields 
\begin{equation*}
h'(\alpha)=-4r\sin\alpha[(1+r^2)\sin\lambda+2r\cos\alpha].
\end{equation*}

\noindent Since $h(\alpha)$ is smooth and periodic, the candidate minimum points of $h(\alpha)$ are the zero points of $h'(\alpha)$ which are $\alpha_{1}=0$, $\alpha_{2}=\pi$ and $\alpha_{3}=\pm \arccos((1+r^2)/2r \sin\lambda)$.
A calculation gives
\[
h(0)=(1+r^2+2r\sin\lambda)^2,
\]
\[
h(\pi)=(1+r^2-2r\sin\lambda)^2,
\]
and 
\[
h(\alpha_{3})=\cos^2 \lambda(1-r^2)^2.
\]
\noindent $\lambda\geq 0$ implies that $h(\pi)\leq h(0)$ and $h(\alpha_{3})<h(0)$ always hold for $0<r<1$. It is easy to check that $h(\pi)<h(\alpha_{3})$ if and only if $\sin\lambda>2r/(1+r^2)$. Therefore we can get the minimum of $h(\alpha)$ are
\begin{equation*}
\begin{cases}
(1+r^2-2r\sin\lambda)^2, \quad \sin \lambda>\frac{2r}{1+r^2},\\
\cos^2\lambda (1-r^2)^2,\quad \sin \lambda\leq \frac{2r}{1+r^2},\\
\end{cases}
\end{equation*}
which can be written in the form
\begin{equation}\label{min}
\begin{cases}
(1+r^2-2r\sin\lambda)^2, \quad r<\tan(\lambda/2),\\
\cos^2\lambda (1-r^2)^2,\quad r\geq \tan(\lambda/2),\\
\end{cases}
\end{equation}
and the maximum of $h(\alpha )$ is 
\begin{equation}\label{max}
(1+r^2+2r\sin\lambda)^2.
\end{equation}
Finally by using $(\ref{h(a)})$, $(\ref{min})$ and $(\ref{max})$, the inequalities can be deduced immediately. 
\end{proof}

\noindent\textit{Proof of Theorem $\ref{theorem 6}$.} The equivalence of (i) and (iv) in Theorem $\ref{theorem 1}$ implies that if $p\in P_{\lambda}$, then there exists a function $\omega(z)\in\mathcal{A}$ with $|\omega(z)|<1$ and $\omega(0)=0$ such that 
\[
p(z)=p_{\lambda}(\omega(z))
\]
in $\mathbb{D}$. 
Then by making use of Lemma $\ref{lemma 4}$ and Schwarz-Pick Lemma, we have 

\begin{equation}\label{inequalities}
\begin{split}
\left| \frac{zp'(z)}{p(z)}\right|
&=
\left|\frac{z\omega'(z)p_{\lambda}'(\omega(z))}{p_{\lambda}(\omega(z))}\right|=\left|\frac{z\omega'}{\omega}\right|\left|\frac{\omega'(z)p'(\omega)}{p(\omega)}\right|\\
&\leq 
\begin{cases}
\dfrac{1-|\omega|^2}{1-|z|^2}\dfrac{2|z|\cos\lambda}{1+|\omega|^2-2|\omega||\sin\lambda|}, \quad |\sin \lambda|>\frac{2|\omega|}{1+|\omega|^2},\\
\dfrac{1-|\omega|^2}{1-|z|^2}\dfrac{2|z|}{1-|\omega|^2},\quad |\sin \lambda|\leq \frac{2|\omega|}{1+|\omega|^2}, 
\end{cases}
\end{split}
\end{equation}

Since $|\omega(z)|\leq|z|=r$, function $2|\omega|/(1+|\omega|^2)$ is increasing in $|\omega|\in[0,r]$ and obtain its maximum value $2r/(1+r^2)$ if and only if  $\omega(z)=xz$ with $|x|=1$. On the other hand,
\[ 
\frac{1-|\omega|^2}{1-|z|^2}\frac{2r\cos\lambda}{1+|\omega|^2-2|\omega||\sin\lambda|}
\]
 is also increasing in $|\omega|$ provided
\[
|\sin \lambda|>\frac{2|\omega|}{1+|\omega|^2}. 
\] 
Therefore inequality \eqref{inequalities} implies

\begin{equation*}
\begin{split}
\left| \frac{zp'(z)}{p(z)}\right|
&\leq
\begin{cases}
\dfrac{2r\cos\lambda}{1+r^2-2r|\sin\lambda|}, \quad |\sin \lambda|>\frac{2r}{1+r^2},\\
\dfrac{2r}{1-r^2},\quad |\sin \lambda|\leq \frac{2r}{1+r^2}.\\
\end{cases}
\\
&=
\begin{cases}
\dfrac{2r\cos\lambda}{1+r^2-2r|\sin\lambda|}, \quad r<|\tan(\lambda/2)|,\\
\dfrac{2r}{1-r^2},\quad r\geq |\tan(\lambda/2)|.
\end{cases}
\end{split}
\end{equation*}

Hence the proof of Theorem $\ref{theorem 6}$ is completed.\qed \\

\par The sharp estimate of the entity in Theorem $\ref{theorem 6}$ first appears in a paper \cite{Ruscheweyh} by  Ruscheweyh and Singh. Their proof was based on variational method.

{\bf Theorem A.}\,
For $p\in\mathcal{P}_{0}$ and $\lambda\in(-\pi/2,\pi/2)$ the estimate
\begin{equation*}
\left|\frac{zp'(z)}{p(z)+i\tan\lambda}\right|\leq
\begin{cases}
\dfrac{(1-|z|^2)\cos\lambda}{1-2|z||\sin\lambda|+|z|^2},\, &|z|<|\tan\frac{\lambda}{2}|,\\
1, \, &|z|\geq |\tan\frac{\lambda}{2}|.
\end{cases}
\end{equation*}
is valid and sharp. Equality holds for certain functions in $\mathcal{P}_{0}$.\\

Note that Theorem $\ref{theorem 6}$ improves Theorem A since it shows that the only extremal functions are $p_{\lambda}(xz)$ with $|x|=1$.

\begin{pro}\label{proposition 1}
Let $p\in \mathcal{P}_{\lambda}$ with $\lambda\in(-\pi/2,\pi/2)$, then
\[
\left|\Im \frac{zp'(z)}{p(z)}\right|\leq M(\lambda,r)
\]
and
\[
\left|\Re \frac{zp'(z)}{p(z)}\right|\leq M(\lambda,r)
\]
where $r=|z|<1$ and $M(\lambda, r)$ is given in \eqref{M2(r)}. Equality occurs at point $z_{0}=re^{i\theta}$ in the first inequality if and only if $p(z)=p_{\lambda}(xz)$ and $r<|\tan(\lambda/2)|$, where $x=e^{i(\alpha-\theta)}$ with $\alpha$ satisfying \eqref{maximum point}.
\end{pro}
\begin{proof}
Since the above inequalities are straightforward consequence of Theorem \ref{theorem 6}, we only need to verify the sharpness. By a simple calculation, it is easy to see that if $\lambda<0$ for any fixed $r<-\tan(\lambda/2)$
\[
\frac{zp'_{\lambda}(z)}{p_{\lambda}(z)}=\frac{z(1+e^{2i\lambda})}{(1-z)(e^{2i\lambda}+z)}=\frac{-2ri\cos\lambda}{1+r^2-2r\sin\lambda}=-iM(\lambda, r)
\]
when $z=ire^{i\lambda}$. Similarly, we can get if $\lambda>0$ for any fixed $r<\tan(\lambda/2)$
\[
\frac{zp'_{\lambda}(z)}{p_{\lambda}(z)}=\frac{2ri\cos\lambda}{1+r^2-2r\sin\lambda}=iM(\lambda,r)
\]
when $z=-ire^{i\lambda}$.
Our proof is completed.
\end{proof}

\par We shall conclude this section with a result due to Kim and Sugawa \cite{KimSugawa} which gives a sufficient condition of $\mathcal{P}_{\lambda}$.
\begin{thm}\label{theorem 7}
Let $p\in \mathcal{A}_{0}$ satisfy that 
\[
\frac{zp'(z)}{p(z)}\prec\frac{zp'_{\lambda}(z)}{p_{\lambda}(z)}
\]
in $\mathbb{D}$, then $p\in\mathcal{P}_{\lambda}$.
\end{thm}

\par Note that the function $zp'_{\lambda}(z)/p_{\lambda}(z)$ maps $\mathbb{D}$ univalently onto $U_{\lambda}$, where $U_{\lambda}$ is the slit domain defined by
\[
U_{\lambda}=\mathbb{C}\setminus \{iy:\, y\geq A_{\lambda} \, \text{or} \,\, y\leq -1/A_{\lambda}\},\,\, A_{\lambda}=\frac{\cos\lambda}{1+\sin\lambda}.
\]

\section{Applications}

\subsection{$\lambda$-spirallike functions}
\begin{definition} (\cite{Duren}, see also \cite{AhujaSilverman})
A function $f\in \mathcal{A}_{1}$ is called \textit{$\lambda$-spirallike} (denoted by $f\in \mathcal{SP}(\lambda)$) for a real number $\lambda\in(-\pi/2,\pi/2)$ if 
\[\frac{zf'}{f}\in \mathcal{P}_{\lambda}.
\]
\end{definition}
\par Spirallike functions are shown to be univalent by {\v S}pa{\v c}ek \cite{Spacek}. Note that $\mathcal{SP}(0)$ is precisely the set of starlike functions normally denoted by $\mathcal{S}^{*}$.

\par By the definition of $\lambda$-spirallike function, we can easily deduce the following corollary from Theorem $\ref{theorem 5}$;
\begin{cor}\label{corollary 1}
Let $f(z)\in \mathcal{SP}(\lambda)$, then
\[
\left|\frac{zf'(z)}{f(z)}-\left(\frac{1}{1-r^2}+\frac{r^2}{1-r^2}e^{-2i\lambda}\right)\right|\leq\frac{2r\cos \lambda}{1-r^2}
\]
where $r=|z|<1$. In particular we have
\[
\frac{1+r^2\cos 2\lambda -2r\cos \lambda}{1-r^2}\leq \Re \frac{zf'(z)}{f(z)}\leq \frac{1+r^2\cos 2\lambda +2r\cos \lambda}{1-r^2}
\]
and
\[
1/A(\lambda,r)\leq\left|\frac{zf'(z)}{f(z)}\right|\leq A(\lambda,r)
\]
where $A(\lambda,r)$ is given by \eqref{A(r)}.
Those inequalities are sharp with extremal functions given by
\[
f_{\lambda}(z)=\frac{z}{(1-z)^{1+e^{-2i\lambda}}}.
\]

\end{cor}

Note that the lower bound of the second estimate was proved by Robertson \cite{Robertson3}, but the others are not given in the literature as far as the author knows.

\subsection{$\lambda$-Robertson functions}

\begin{definition}
A function $f\in \mathcal{A}$ is said to be a \textit{$\lambda$-Robertson function} (denoted by $f\in\mathcal{R}(\lambda)$) if $f$ satisfies
\[
1+\frac{zf''(z)}{f'(z)}\in \mathcal{P}_{\lambda}
\]
for all $z \in \mathbb{D}$. 
\end{definition}
Note that $\mathcal{R}(0)$ is precisely the set of convex functions which is sometimes denoted by $\mathcal{K}$. $\lambda$-Robertson functions were first introduced by Robertson \cite{Robertson4} in 1969 and have been investigated by various authors, see for example \cite{HottaWang} and the references therein. Among other properties, the univalence is of interest and is shown by Pfaltzgraff \cite{Pfaltzgraff} that all the functions in $\mathcal{R}(\lambda)$ are univalent if and only if $\lambda=0$ or $|\lambda| \in [\pi/3, \pi/2)$.

In \cite{KimSrivastava}, Kim and Srivastava posed the \textit{open problem} whether
\[
1+\frac{zf''(z)}{f'(z)}\prec 1+\frac{zf_{\lambda}''(z)}{f_{\lambda}'(z)}=p_{\lambda}(z)
\]
implies 
\[
\frac{zf'(z)}{f(z)}\prec \frac{zf_{\lambda}'(z)}{f_{\lambda}(z)}
\]
for $z\in\mathbb{D}$ where $f_{\lambda}(z)=((1-z)^{1-2e^{-i\lambda}\cos\lambda}-1)/(2e^{-i\lambda}\cos\lambda -1)$.
In fact, it is an extension of the case $\lambda=0$ which was proved by MacGregor \cite{MacGregor} in 1975.
In a forthcoming paper of the present author, we give an example to show that it is not true for general $\lambda\not=0$ and also show that it holds in a restricted disc. As a corollary, we can obtain the radius of spirallikeness of $\lambda$-Robertson functions. We shall only state this result without proof. For more information, the reader may be refered to \cite{Wang2}.

\begin{thm}[See \cite{Wang2}]
Let $f\in \mathcal{R}(\lambda)$, then
\[
\Re e^{i\lambda}\frac{zf'(z)}{f(z)}>0
\]
in $|z|<R(\lambda)$ where $R(\lambda)$ is given by
\[
R(\lambda)=\sup \left\{r<1\,: \sup_{z\in\mathbb{D}} \left|\frac{1}{r}\frac{mrz-1+(1-rz)^{m}}{1-(1-rz)^m}\right|<1\right\}
\] 
where $m=1+e^{-2i\lambda}$. This result is sharp.

\end{thm}

\begin{remark}
It is easy to verify that 
\[
R(0)=1
\]
which is the result of MacGregor \cite{MacGregor}.
\end{remark}

\subsection{Close-to-convex functions with argument $\lambda$}
\begin{definition}
A function $f\in \mathcal{A}_{1}$ is said to be \textit{close-to-convex} (denoted by $f\in \mathcal{CL}$) if there exist a starlike function $g$ and a real number $\lambda\in(-\pi/2,\pi/2)$ such that
\[
\frac{zf'}{g}\in \mathcal{P}_{\lambda}.
\]
\end{definition} 
\par If we specify the real number $\lambda$ in the above definition, the corresponding function is called a \textit{close-to-convex function with argument $\lambda$} and we denote the class of these functions by $\mathcal{CL}(\lambda)$ (see [5, II, Definition 11.4]). Note that the union of class $\mathcal{CL}(\lambda)$ over $\lambda\in(-\pi/2, \pi/2)$ is precisely $\mathcal{CL}$. The sharp coefficients bounds of the class $\mathcal{CL}(\lambda)$ are known (see \cite{Wang}).

\begin{lem}[See \cite{Duren}]\label{lemma 5}
Let $f\in\mathcal{S}^*$, then
\[
\frac{r}{(1+r)^2}\leq|f(z)|\leq\frac{r}{(1-r)^2}
\]
where $r=|z|$. Equality occurs if and only if $f$ is a suitable rotation of the Koebe function $k(z)=z/(1-z)^2$.

\end{lem}

\par By applying Theorem $\ref{theorem 5}$ and Lemma $\ref{lemma 5}$, we can get the sharp distortion theorem for $\mathcal{CL}(\lambda)$.

\begin{thm}\label{theorem 8}
Let $f(z)\in \mathcal{CL}(\lambda)$ for a real constant $\lambda\in(-\pi/2,\pi/2)$, then
\[
\frac{1}{A(\lambda,r)(1+r)^2}\leq |f'(z)|\leq \frac{A(\lambda,r)}{(1-r)^2}.
\]
where $A(\lambda,r)$ is given in $(\ref{A(r)})$ and $r=|z|<1$. The inequalities are sharp with extremal functions $f(z)$ satisfying
\[
f'(z)=\frac{1+e^{-2i\lambda}xz}{(1-yz)^2(1-xz)}
\]
for $|x|=|y|=1$.
\end{thm}

\begin{remark}
Theorem $\ref{theorem 8}$ improves the distortion theorem of close-to-convex functions (see \cite{Duren}) since the real-valued function $A(\lambda,r)$ is symmetric in $\lambda$ with respect to the origin and 
\[
\frac{1-r}{1+r}\leq A(\lambda,r)\leq \frac{1+r}{1-r}
\]
for any $\lambda\in(-\pi/2,\pi/2)$.
\end{remark}

Note that it is easy to deduce the growth theorem of close-to-convex functions with argument $\lambda$ from Theorem \ref{theorem 8}, we omit it here since the form is not very pleasible.

\subsection{Analytic functions whose derivative is in $\mathcal{P}_{\lambda}$}

\paragraph{} \quad  Let $\mathcal{D}(\lambda)=\{f\in\mathcal{A}_{1}:\, f'\in\mathcal{P}_{\lambda}\}$ for $-\pi/2<\lambda<\pi/2$. It is easy to see that $\mathcal{D}(\lambda)\subset\mathcal{CL}(\lambda)$, thus $\mathcal{D}(\lambda)\subset\mathcal{S}$. Some properties of $\mathcal{D}(\lambda)$ can be deduced by those of $\mathcal{D}(0)$ which have been studied in \cite{Gray-Ruscheweyh}, \cite{MacGregor-1962} and so on. We shall only present a distortion theorem which is a direct consequence of Theorem \ref{theorem 5}.

\begin{thm}
Let $f\in\mathcal{D}(\lambda)$, then
\[
1/A(\lambda,r)\leq |f'(z)|\leq A(\lambda,r),
\]
where $r=|z|<1$ and $A(\lambda,r)$ is given by \eqref{A(r)}. These inequalities are sharp with extremal function
\begin{equation}\label{extremal f}
f(z)=-(1+e^{-2i\lambda})\log(1-z)-e^{-2i\lambda}z.
\end{equation}
\end{thm}
For a locally univalent function $f$, the hyperbolic norm of the Pre-Schwarzian derivative $T_{f}=f''/f'$ is defined by on $\mathbb{D}$
\[
||f||=\sup_{|z|<1}(1-|z|^2)|T_{f}(z)|.
\]
Since each function in $\mathcal{P}_{\lambda}$ is a Gelfer function, by Gelfer's theorem (see [14, Theorem 2.4]), we have for each function $f\in\mathcal{D}(\lambda)$,
\[
||f||\leq 2.
\]
Our next result shows that this estimate is sharp for the class $\mathcal{D}(\lambda)$, and the extremal functions are also given.
\begin{thm}\label{theorem 11}
Let $f\in\mathcal{D}(\lambda)$, then 
\[
||f||\leq 2.
\]
This bound is sharp for each $\lambda\in(-\pi/2,\pi/2)$ with extremal function $f$ given in \eqref{extremal f}
\end{thm}
\begin{proof}
For $f\in\mathcal{D}(\lambda)$, we have $f'\in\mathcal{P}_{\lambda}$, thus in view of Theorem \ref{theorem 6}, 
\[
\left|\frac{zf''(z)}{f'(z)}\right|\leq M(\lambda, |z|)
\]
where $M(\lambda, r)$ is given in \eqref{M2(r)}. Remark 3 gives that $M(\lambda,r)\leq M(0,\lambda)=2r/(1-r^2)$, therefore we have $||f||\leq 2$. The sharpness can be obtained by observing that $M(\lambda,r)=M(0,\lambda)$ if $r>|\tan(\lambda/2)|$.

\end{proof}

Note that the hyperbolic norm of $f$, $f\in\mathcal{D}(0)$ was obtained by Nunokawa \cite{Nunokawa} as well. It is known that  (cf. \cite{Kim-Sugawa}) $f$ is bounded if $||f||<2$ and the bound depends only on the value of $||f||$. Therefore Theorem \ref{theorem 11} implies that every function not of the form $f(xz)/z$ where $f(z)$ is given in \eqref{extremal f} and $|x|=1$ in $\mathcal{D}(\lambda)$ is bounded.

\[
\text{Acknowledgements}
\]
The author is grateful to Professor Toshiyuki Sugawa for his constant encouragement and useful discussions during the preparation of this paper. Without his help, this paper will never be finished.



\begin{thebibliography}{99}

\bibitem{AhujaSilverman}
O.~P. Ahuja and H.~Silverman, \emph{A survey on spiral-like and related
  function classes,}
 Math. Chronicle \textbf{20} (1991), 39--66.

\bibitem{Bernardi}
S. D. Bernardi,
\textit{New distortion theorems for functions of positive real part and applications to the partial sums of univalent convex functions,}
Proc. Amer. Math. Soc. \textbf{45} (1974), 113-118.



\bibitem{Duren}
P. L. Duren,
\textit{Univalent functions,}
Grundlehren Math. Wiss. 259, Springer-Verlag, New York, 1983.



\bibitem{Gelfer}
S. Gelfer,  
\textit{On the class of regular functions which do not take on any pair of values $\omega$ and $-\omega$,}
 Mat. Sb. \textbf{19} (1946), 33-46 (Russian).



\bibitem{Good}
A. W. Goodman,  
\textit{Univalent Functions,} 
2 vols., Mariner Publishing Co. Inc., 1983.


\bibitem{GrahamKohr}
I. Graham and G. Kohr,
\textit{Geometric function theory in one and high dimensions,}
Marcel Dekker, Inc., New York, 2003.



\bibitem{Gray-Ruscheweyh}
F. Gray and S. Ruscheweyh,
\textit{functions whose derivatives take values in a half-plane,}
Proc. Amer. Math. Soc. \textbf{104} (1988), 215-218.

\bibitem{DH}
D. J. Hallenbeck and T. H. MacGregor,
\textit{Linear problems and convexity techniques in geometric function theory,}
Pitman(1984).




\bibitem{He}
G. Herglotz,
\textit{\"Uber Potenzreihen mit positivem, reellen Teil in Einheitskreis,}
Ber. Verh. Sachs. Akad. Wiss. Leipzig(1911) 501-511.

\bibitem{Holland}
F. Holland,
\textit{The extremal points of a class of functions with positive real part,}
Math. Ann. \textbf{202} (1973) 85-88.


\bibitem{HottaWang}
I. Hotta and L.-M. Wang,
\textit{Boundedness, univalence and quasiconformal extension of Robertson functions,}
preprint (arXiv:1003.2019).


\bibitem{KimSrivastava}
Y. C. Kim and H. M. Srivastava,
\textit{Some subordination properties for spirallike functions,}
Appl. Math. Comput.d \textbf{203} (2008), 838-842.


\bibitem{Kim-Sugawa}
Y. C. Kim and T. Sugawa,
\textit{Growth and coefficient estimates for uniformly locally univalent functions of the unit disk,}
Rocky Mountain J. Math. \textbf{32} (2002), 179-200.

\bibitem{Sugawa}
\bysame,
\textit{A conformal invariant for nonvanishing analytic functions and its applications,}
Michigan Math. J. \textbf{54} (2006).

\bibitem{KimSugawa2006}
\bysame,
\textit{Norm estimates of the pre-Schwarzian derivative for certain classes of univalent functions,}
Proc. Edinburgh Math. Soc, \textbf{49} (2006), 131-143

\bibitem{KimSugawa}
\bysame,
\textit{ A note on Bazilevi{\v c} functions,}
Taiwanese J. Math. \textbf{13} (2009), 1489-1495.


\bibitem{Kirwan}
W. E. Kirwan II,  
\textit{Extremal problems for certain classes of analytic functions,} 
Ph.D. Thesis, Rutgers University, New Brunswick, N. J., 1964.


\bibitem{Kortram}
R. A. Kortram,
\textit{The extremal points of a class of functions with positive real part,}
Bull. Belg. Math. Soc. \textbf{4} (1997), 449-459.


\bibitem{Libera}
R. J. Libera,
\textit{Some radius of convexity problems,}
Duke Math. J. \textbf{31} (1964), 143-158.


\bibitem{MacGregor-1962}
T. H. MacGregor,
\textit{Functions whose derivative has a positive real part,}
Trans. Amer. Math. Soc. \textbf{104} (1962), 532-537.

\bibitem{MacGregor}
\bysame,
\textit{A subordination for convex functions of order a,}
J. London Math. Soc. (2), \textbf{9} (1975), 530-536.


\bibitem{Nunokawa}
M. Nunokawa,
\textit{On the univalency and multivalency of certain analytic functions,}
Math. Z. \textbf{104} (1986), 394-404.


\bibitem{Pfaltzgraff}
J.~A. Pfaltzgraff, \emph{Univalence of the integral of {$f^{\prime}
  (z)^{\lambda}$},} Bull. London Math. Soc. \textbf{7} (1975), no.~3, 254--256.

\bibitem{Robertson1}
M. S. Robertson,
\textit{Variational methods for univalent functions with positive real part,}
Trans. Amer. Math. Soc. \textbf{102} (1962), 82-93.

\bibitem{Robertson2}
\bysame,
\textit{Extremal problems analytic functions with positive real part and applications,}
Trans. Amer. Math. Soc. \textbf{106} (1963), 236-253.


\bibitem{Robertson3}
\bysame,
\textit{Radii of star-likeness and close-to-convexity,}
Proc. Amer. Math. Soc. \textbf{16} (1965), 847-852.

\bibitem{Robertson4}
\bysame, \emph{Univalent functions {$f(z)$} for which {$zf^{\prime}
  (z)$} is spirallike,} Michigan Math. J. \textbf{16} (1969), 97--101.



\bibitem{Ru}
S. Ruscheweyh,
\textit{Convolution in geometric function theory,}
S{\'e}m. Math. Sup. 83, University of Montr{\'e}al, Montr{\'e}al, 
Qu{\'e}bec, Canada 1982.

\bibitem{Ruscheweyh}
S. Ruscheweyh and V. Singh,
\textit{On certain extremal problems for functions with positive real part,}
Proc. Amer. Math. Soc. \textbf{61} (1976), 329-334.

\bibitem{Sakaguchi}
K. Sakaguchi,
\textit{A variational method for functions with positive real part,}
J. Math. Soc. Japan \textbf{16} (1964), 287-297.

\bibitem{Schur}
I. Schur,
\textit{\"Uber Potenzreihen, die im Innern des Einheitskreises beschr\"ankt
sind,}
J. Reine Angew. Math. \textbf{147} (1917) 205-232; \textbf{148} (1918)
122-145.


\bibitem{Spacek}
L. {\v S}pa{\v c}ek,
\textit{Contribution {a} la th{\'e}orie des fonctions univalentes,}
{\v C}asopis P{\v e}st. Mat.-Fys. \textbf{62} (1932), 12-19.


\bibitem{Wang}
L.-M. Wang,
\textit{Coefficient estimates for close-to-convex functions with argument $\lambda$,}
Bull. Berg. Math. Soc. (to appear).



\bibitem{Wang2}
\bysame,
\textit{Subordination problems of Robertson functions,}
preprint (arXiv:1006.0547).


\bibitem{Yamashita}
S. Yamashita,
\textit{Gelfer functions, integral means, bounded mean oscillation, and univalency,}
Trans. Amer. Math. Soc. \textbf{321} (1990), 245-259.


\bibitem{Zmorovic1}
V. A. Zmorovi{\v c},
\textit{On bounds of convexity for starlike function of order $\alpha$ in the circle $|z|<1$ and in the circular region $0<|z|<1$,}
Mat. Sb. \textbf{68} (110) (1965), 518-526; English transl., Amer. Math. Soc. Transl. (2) \textbf{80} (1969), 203-213.



\bibitem{Zmorovic2}
\bysame,
\textit{On the bounds of starlikeness and univalence in certain classes of functions regular in the circle $|z|<1$,}
Ukrain. Mat. {\v Z}. \textbf{18}(1966), 28-39; English transl., Amer. Math. Soc. Transl. (2) \textbf{80} (1969), 227-242.



\end{thebibliography}
\end{document}